\documentclass[12pt]{amsart}

\usepackage{amssymb}
\usepackage{bm}
\usepackage{graphicx}
\usepackage{psfrag}
\usepackage[usenames,dvipsnames]{xcolor}
\usepackage{url}
\usepackage{enumerate}
\usepackage{float}
\usepackage{algpseudocode}

\usepackage{mathtools}

\usepackage{xy}
\input xy
\xyoption{all}

\usepackage{blkarray}

\newcommand{\excise}[1]{}

\newtheorem{thm}{Theorem}[section]
\newtheorem{lemma}[thm]{Lemma}

\newtheorem{prop}[thm]{Proposition}

\theoremstyle{definition}

\newtheorem{example}[thm]{Example}

\numberwithin{equation}{section}

\newcommand{\ring}[1]{\ensuremath{\mathbb{#1}}}

\renewcommand\>{\rangle}
\newcommand\<{\langle}

\newcommand\QQ{\ring{Q}}
\newcommand\RR{\ring{R}}

\newcommand\ZZ{\ring{Z}}

\newcommand\ww{{\mathbf w}}
\newcommand\al{{\boldsymbol\alpha}}
\newcommand\be{{\boldsymbol\beta}}

\newcommand\ee{{\mathbf e}}
\newcommand\vv{{\mathbf v}}
\newcommand\LL{{\mathbb L }}

\newcommand\hh{{\mathbf h}}
\newcommand\rr{{\mathbf r}}
\newcommand\uu{{\mathbf u}}

\def\abs#1{\left\vert{#1}\right\vert}

\def\NPP#1{{\mathcal{O}_{#1}^{-}}}
\def\PPN#1{{\mathcal{O}_{#1}^{+}}}
\def\PNP#1{{\mathcal{O}_{#1}^{\pm}}}

\def\PPNhat#1{{\widehat{\mathcal{O}}_{#1}^{+}}}
\def\PNPhat#1{{\widehat{\mathcal{O}}_{#1}^{\pm}}}
\def\SSplus#1{{\mathcal{S}}_{#1}^+}
\def\SSminus#1{{\mathcal{S}}_{#1}^-}
\def\SSpm#1{{\mathcal S_{#1}^{\pm}}}

\begin{document}

\mbox{}
\title[Graver bases of shifted numerical semigroups with 3 generators]{Graver bases of shifted numerical semigroups \\ with 3 generators}

\author[Howard]{James Howard}
\address{Mathematics Department\\San Diego State University\\San Diego, CA 92182}
\email{jhoward8987@sdsu.edu}

\author[O'Neill]{Christopher O'Neill}
\address{Mathematics Department\\San Diego State University\\San Diego, CA 92182}
\email{cdoneill@sdsu.edu}

\date{\today}

\begin{abstract}
A numerical semigroup $M$ is a subset of the non-negative integers that is closed under addition. A factorization of $n \in M$ is an expression of $n$ as a sum of generators of $M$, and the Graver basis of $M$ is a collection $Gr(M_t)$ of trades between the generators of $M$ that allows for efficient movement between factorizations.  Given positive integers $r_1, \ldots, r_k$, consider the family $M_t = \<t + r_1, \ldots, t + r_k\>$ of ``shifted'' numerical semigroups whose generators are obtained by translating $r_1, \ldots, r_k$ by an integer parameter $t$.  In this paper, we characterize the Graver basis $Gr(M_t)$ of $M_t$ for sufficiently large $t$ in the case $k = 3$, in the form of a recursive construction of $Gr(M_t)$ from that of smaller values of $t$.  As a consequence of our result, the number of trades in $Gr(M_t)$, when viewed as a function of $t$, is eventually quasilinear.  We~also obtain a sharp lower bound on the start of quasilinear behavior.  
\end{abstract}

\maketitle



\section{Introduction}
\label{sec:intro}

A \emph{numerical semigroup} is a cofinite subset $M \subseteq \ZZ_{\ge 0}$ that is closed under addition and contains $0$ (see~\cite{numerical} for a thorough introduction).  We often specify a numerical semigroup using generators $n_1 < \cdots < n_k$,~i.e.,
\[
M = \<n_1, \ldots, n_k\> = \{z_1 n_1 + \cdots + z_k n_k : z_i \in \ZZ_{\ge 0}\}.
\]
A \emph{factorization} of an element $n \in M$ is an expression 
$$n = z_1 n_1 + \cdots + z_k n_k$$
of $n$ as a sum of generators of $S$, which we often encode as a $k$-tuple $(z_1, \ldots, z_k)$.  
One of the primary ways the factorization structure of $M$ is studied is via \emph{trades} between the generators of $M$, encoded as $k$-tuples $(z_1, \ldots, z_k) \in \ZZ^k$ satisfying 
$$z_1 n_1 + \cdots + z_k n_k = 0.$$

Numerical semigroups arise in countless areas of mathematics~\cite{numericalappl}.  In the present work, we focus on their appearance in discrete optimization in the form of \emph{knapsack problems}, which, in the language above, seeks the factorization of a given element $n \in M$ that optimizes some linear functional~\cite{knapsacksurvey}.  One of the primary computational tools for solving linear integer programs of this form is the \emph{Graver basis} of $M$, which is the collection $Gr(M) \subseteq \ZZ^k$ of trades of $M$ that cannot be written as a sum of nonzero trades with identical sign pattern 
(see~\cite{alggeodiscopt} for a thorough treatment).  
If one computes the Graver basis for $M$, then any knapsack problem for $M$ can be solved quickly by locating some solution and then greedily applying trades in the Graver basis that yield until the optimum is reached.  Unfortunately, Graver bases are generally quite large, making them infeasible for large systems~\cite{nfold}.  

Numerical semigroups with 3 generators are a common first-case to consider, and much more is known for this family than in general~\cite{lengthdist1,deltadim3nonsym,deltadim3sym}, due in part to their restricted trade structure~\cite{factordim3} that fails to extend to numerical semigroups with 4 generators or more.  Despite this, Graver bases are one of the few properties numerical semigroups that lack a characterization even in the case of 3-generators.  

One particularly well-studied family of numerical semigroups are those of the form $\<t, t + r_1, \ldots, t + r_k\>$ for $t$ large and $r_1, \ldots, r_k \in \ZZ_{\ge 1}$ fixed.  Many attributes of numerical semigroups that are unwieldy in general admit concise formulas and descriptions for semigroups in this family, including Frobenius number and genus~\cite{shiftyapery}, delta set~\cite{shiftydelta}, catenary degree~\cite{shiftyminpres}, and Betti numbers~\cite{vu14}, with some additional results for numerical semigroups with 3 generators~\cite{shiftedtangentcone,shifted3gen}.  
The goal of this manuscript is to add Graver bases to this list.  

Let us briefly examine a related result.  A Markov basis (also known as a minimal presentation) of a numerical semigroup $S$ is a collection $R$ of trades such that for any $n \in S$, any two factorizations of $n$ can be connected by a sequence of factorizations of $n$ in which successive factorizations differ by a trade in $R$~\cite{markovbook,algmarkov}.  Markov bases can be used to solve knapsack problems, and often contain far less trades than the Graver basis, but their reduced size comes at a cost of efficiency.  It was proven in~\cite{shiftyminpres} that there is a bijection between the Markov bases of $M_t$ and those of $M_{t+r_k}$ for $t$ sufficiently large.  In particular, Markov basis cardinality of $M_t$, when viewed as a function of $t$, is eventually periodic.  Note that the cardinality of $Gr(M_t)$, on the other hand, is unbounded as $t$ grows whenever $M_t$ has at least 3 minimal generators.  

In this paper, we examine the Graver bases of shifted families of numerical semigroups with 3 generators, which in general take the form
$$M_t = \<t - da, t, t + db\>$$
with $a, b, d \in \ZZ_{\ge 1}$ and $\gcd(a,b) = 1$.  
Our main result is the following theorem, obtained as a consequence of a geometric characterization of the Graver basis of $M_t$ for large $t$.  
Note the difference in asymptotic growth rate as compared to the cardinatlity of the Markov basis discussed above.  

\begin{thm}\label{t:graversize}
	For $t \gg 0$, $|Gr(M_{t})|$ coincides with a quasilinear function of $t$ with period $\rho = dab(a+b)$ and leading coefficient $2/ab$.  In particular,
	$$|Gr(M_{t + \rho})| = |Gr(M_t)| + 2d(a+b)$$
	for all $t  \gg 0$.
\end{thm}

Our methods are constructive:\ we identify explicit bijections between portions of $Gr(M_t)$ and $Gr(M_{t+\rho})$ that allow one to compute the latter Graver basis directly from the former.  The bijective maps are constructed in Sections~\ref{sec:pnp} and~\ref{sec:ppn}, after setting up notation in Section~\ref{sec:prelims}, and their utility is illustrated in an example given alongside the proof of Theorem~\ref{t:graversize} in Section~\ref{sec:graverbasis}.  

Additionally, the lower bound we obtain on $t$ in the proof of Theorem~\ref{t:graversize} is sharp in all cases (see Propositions~\ref{p:hirreducible} and~\ref{p:lendexists}).  In fact, our proof examines each orthant of $\RR^3$ separately, and we obtain a sharp lower bound on $t$ within each orthant.  Note that this is unusual for eventual-quasipolynomiality results in this area; for instance, the lower bound for the results in~\cite{vu14} discussed above concerning Markov bases of shifted numerical semigroups was improved in~\cite{shifted3gen} for 3-generated numerical semigroups, and subsequently improved further in~\cite{homogeneousnumerical}, but a sharp bound remains elusive.  



\section{Preliminaries and Setup}
\label{sec:prelims}

Fix $a, b, d \in \ZZ_{\ge 0}$ with $\gcd(a,b) = 1$, and let $\rr = (-a, 0, b)$.  For each $t > da$ with $\gcd(t,d) = 1$, let
$$M_t = \<t - da, t, t + db\>$$
denote a (not necessarily minimally generated) numerical semigroup. 
The \emph{factorization homomorphism} $\pi_t:\ZZ^3 \to M_t$ is given by
\[
\pi_t(\vv) = (t - da)v_0 + tv_1 + (t + db)v_2
\]
for each $\vv = (v_0, v_1, v_2) \in \ZZ^3$.  In particular, if $\vv \in \ZZ_{\ge 0}^3$, then the element $\pi_t(\vv)$ has $\vv$ has a factorization.  
The kernel 
$$\LL_t = \ker \pi_t$$
is called the \emph{trade lattice} of $M_t$, and its elements are called \emph{trades} of $M_t$.  The coordinate sum of $\vv \in \ZZ^3$ is known as its \emph{length}, which we denote
\[\ell(\vv) = v_0 + v_1 + v_2.\]
A trade $\vv \in \LL_t$ with length $\ell(\vv) = 0$ is called \emph{homogeneous}.  The expressions
\setcounter{equation}{-1}
\begin{align}
	\label{eqn:0} \pi_t(\vv) &= (t-da)\ell(\vv) + dav_1 + d(a+b)v_2 \\ 
	\label{eqn:1} \pi_t(\vv) &= -dav_0 + t\ell(\vv) + dbv_2 \\
	\label{eqn:2} \pi_t(\vv) &= -d(a+b)v_0 -dbv_1 + (t+db)\ell(\vv),
\end{align}
each obtained using $\ell(\vv)$ to eliminate some $v_i$, occur frequently throughout the paper.  
The record the following basic facts from~\cite[Proposition~2.9]{delta} and \cite[Proposition~5]{samelengthfactorizations}, will also be useful throughout the paper.  

\begin{lemma}\label{l:basiclengthfacts}
The following hold.  
	\begin{enumerate}[(a)]
	\item 
	Any $\vv \in \LL_t$ satisfies $\ell(\vv) \in d\ZZ$.

	\item 
	Every homogeneous trade in $\LL_t$ is an integer multiple of
	\[
	\hh = (b, -(a+b), a),
	\]
	which, in particular, is independent of $t$.  
	\end{enumerate}
\end{lemma}


Vectors in the interior of an orthant of $\RR^3$ intersecting $\LL_t$ nontrivially must have exactly 2 entries with equal sign.  
We denote the intersection of $\LL_t$ with each orthant with exactly two non-negative entries by
\[
	\PPN{t} = \{\vv \in \LL_t \mid v_0, v_1 \ge 0\}, \,\,\,
	\PNP{t} = \{\vv \in \LL_t \mid v_0, v_2 \ge 0\}, \,\,\,
	\NPP{t} = \{\vv \in \LL_t \mid v_1, v_2 \ge 0\}.
\]
As we will see in Propositions~\ref{p:pnplengths} and~\ref{p:ppnlengths}, each superscript indicates the possible signs of lengths of trades therein.  We also use symmetry to our advantage:\ since the negation of any trade is also a trade, the above sets and their negations cover $\LL_t$.  

Let $C \subseteq \RR^d$ be a rational cone (i.e., $C$ equals the non-negative $\RR$-span of finitely many vectors in $\QQ^d$) that is pointed (i.e, $C$ contains no nontrivial $\RR$-linear subspaces), and let $S = C \cap \ZZ^d$.  The \emph{Hilbert basis} of $S$ is the set $\mathcal H(S)$ of minimal generators of~$S$ under addition.  It is known that $\mathcal H(S)$ is finite.  
The \emph{Graver basis} of the numerical semigroup $M_t$ is the set $Gr(M_t) \subseteq \LL_t$ of trades that cannot be written as a sum of nonzero trades in the same orthant.  In other words,
$$
Gr(M_t) = 
\!\!\!\! \bigcup_{s \in \{+,\pm,-\}}\!\!\!\! \mathcal H(\mathcal O_t^s) \cup \mathcal H(-\mathcal O_t^s)
$$

As we look towards the proof of Theorem~\ref{t:graversize}, namely that $|Gr(M_t)|$ coincides with a quasilinear function of $t$ for sufficiently large $t$, we denote by
$$
\rho = dab(a + b)
\qquad \text{and} \qquad
B = \max\{B^+, B^\pm, B^-\}
$$
the \emph{period} of the quasilinear function and the lower bound on $t$, respectively, where
\[
	B^+ = (b-1)(a+b) - b(d+1),
	\quad
	B^{\pm} = dab,
	\quad \text{and} \quad
	B^- = (a-1)(a+b) - a(d-1)
\]
denote the lower bounds on $t$ we will obtain for orthant-specific results.  

To prove Theorem~\ref{t:graversize}, we examine each of the orthants $\PPN{t}$, $\PNP{t}$, and $\NPP{t}$, in turn, doing the following for $t > B$:\ 
\begin{enumerate}[(i)]
\item 
we establish a bijection between the subsets of $Gr(M_t)$ and $Gr(M_{t+\rho})$ that lie within one $\hh$-neighborhood of the coordinate hyperplanes; and

\item 
we argue that no trades in $\mathcal H(\PNP{t})$ lie outside of these $\hh$-neighborhoods, and that in $\mathcal H(\PPN{t})$ and $\mathcal H(\NPP{t})$, all remaining trades lie on a line segment with fixed lengths $d$ and $-d$, respectively.  

\end{enumerate}

As an example, let $a = 2$, $b = 3$, and $d = 1$.  Figure~\ref{fig:orthants} depicts a projection into $\RR^2$ of the trade lattice of $M_{49} = \<47,49,52\>$.  Points indicate trades in $Gr(M_{49})$, and the thick dividing lines between the 6 regions each demarcate the intersection with a coordinate hyperplane.  Thin blue strips near these dividing lines indicate the $\hh$-neighborhoods around the coordinate hyperplanes. 


The maps defined in the following theorem will comprise the bijections mentioned in item~(i) above.  We record some basic facts before beginning our examination of the central orthant, $\PNP{t}$, in the next section.  
Here, we denote the standard basis vectors
$$
\ee_0 = (1,0,0), 
\qquad
\ee_1 = (0,1,0),
\qquad \text{and} \qquad
\ee_2 = (0,0,1).
$$

\begin{figure}[t]
\begin{center}
\includegraphics[height=3.5in]{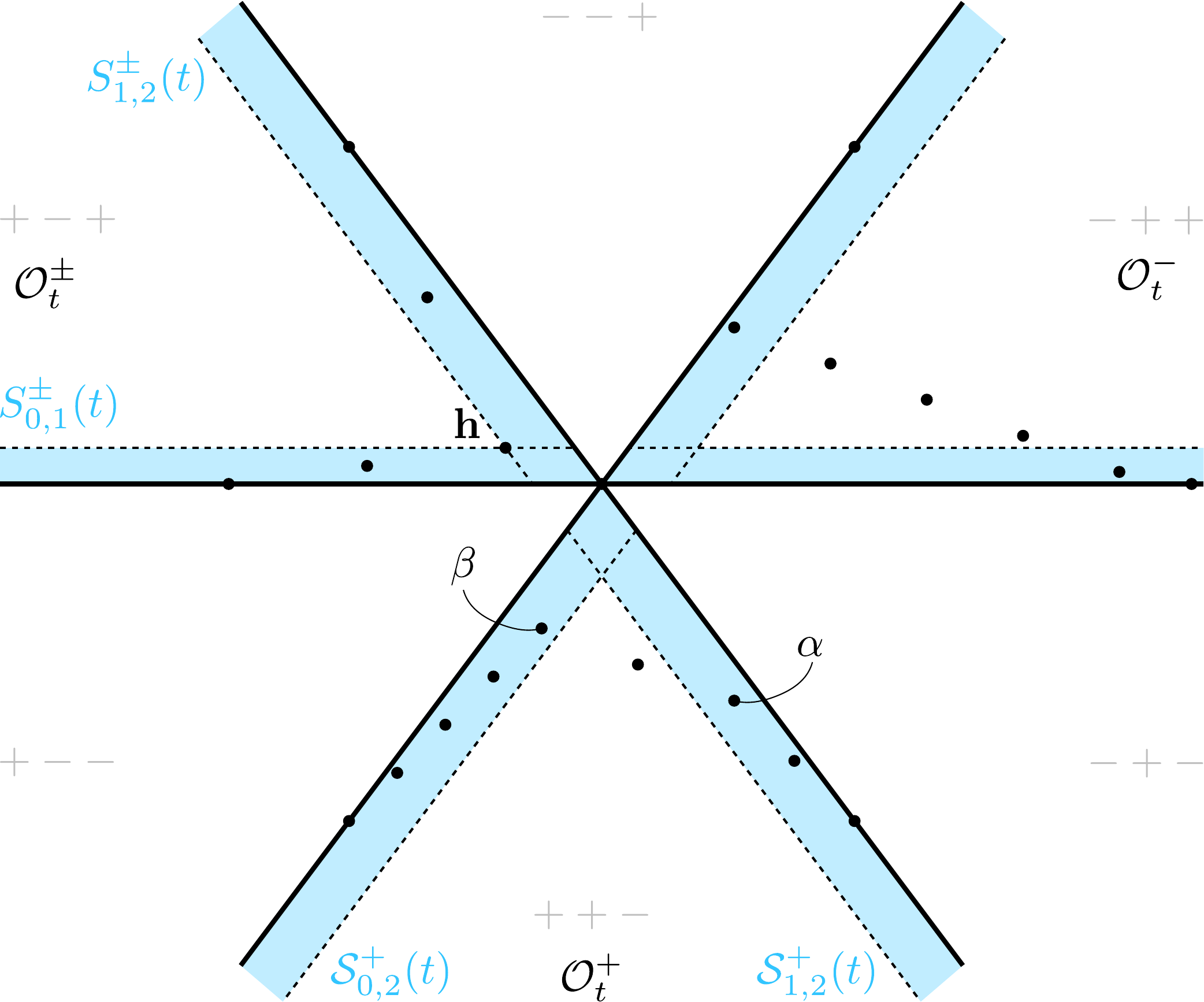}
\end{center}
\caption{The trades in the Graver basis of $M_t = \<47,49,52\>$.  }
\label{fig:orthants}
\end{figure}

\begin{thm}\label{t:big_phi_properties}
	For distinct $i, j \in \{0,1,2\}$, 
	\[
	\varphi_{i,j}:\LL_t \to \LL_{t+\rho}
	\qquad \text{given by} \qquad
	\vv \longmapsto \vv + \frac{\rho\ell(\vv)(\ee_i - \ee_j)}{d(r_j - r_i)}
	\]
	is a well-defined, bijective linear map that preserves length, that is, 
	$\ell(\vv) = \ell(\varphi_{i,j}(\vv))$.  
\end{thm}

\begin{proof}
	Linearity is immediate since $\varphi_{i,j}$ is a linear combination of linear maps, and length preservation follows from the linearity of $\ell$ and the fact that $\ell(\ee_i - \ee_j) = 0$.
	To~ensure $\varphi_{i,j}$ is well-defined, our task is to verify $\varphi_{i,j}(\LL_t) \subseteq \LL_{t+\rho}$.  For distinct $i$ and~$j$, we claim $d(r_j - r_i) \mid \rho$. Indeed, we see that
	\[\rho = dab(a+b) = b(a+b)\left(d(r_1 - r_0)\right) = ab\left(d(r_2 - r_0)\right) = a(a+b)\left(d(r_2 - r_1)\right),\]
	ensuring that integer vectors map to integer vectors. Furthermore, 
	\begin{equation}\label{eq:pitrho}
		\pi_{t+\rho}(\vv) = (t-da)v_0 + tv_1 + (t+db)v_2 + \rho(v_0+v_1+v_2)= \pi_t(\vv) + \rho\ell(\vv),
	\end{equation}
	and thus
	\begin{align*}
		\pi_{t+\rho}(\varphi_{i,j}(\vv)) - \pi_t(\vv)
		&= \pi_{t}(\varphi_{i,j}(\vv)) + \rho\ell(\varphi_{i,j}(\vv)) - \pi_t(\vv)
		= \pi_t(\varphi_{i,j}(\vv) - \vv) + \rho\ell(\vv) \\
		&= \rho\ell(\vv)\pi_t\left(\frac{\ee_i - \ee_j}{d(r_j - r_i)}\right) + \rho\ell(\vv)
		= 0
	\end{align*}
	using the fact that $\ell(\ee_i - \ee_j) = 0$.  As such, if $\vv \in \LL_t$, then $\pi_{t+\rho}(\varphi_{i,j}(\vv)) = \pi_t(\vv) = 0$.
	Injectivity is clear by linearity since $\ker \varphi_{i,j} = 0$.  This means $\varphi_{i,j}$ is full-rank as an $\RR$-linear map, whose inverse is given by 
	\[
	\vv \longmapsto \vv - \frac{\rho\ell(\vv)(\ee_i - \ee_j)}{d(r_j - r_i)}.
	\]
	By equation~\eqref{eq:pitrho} and the linearity of $\pi$, 
	\[
	\pi_t(\varphi_{i,j}^{-1}(\vv)) = \pi_t(\vv) - \rho\ell(\vv)\pi_t\left(\frac{\ee_i - \ee_j}{d(r_j - r_i)}\right) = \pi_t(\vv) + \rho\ell(\vv) = \pi_{t+\rho}(\vv),
	\]
	which implies $\varphi_{i,j}$ is a bijection.  
\end{proof}

\section{The Hilbert Basis of $\PNP{t}$}
\label{sec:pnp}

In this section, we consider the central orthant, $\PNP{t}$.  Theorem~\ref{t:pnpgraver} identifies a bijection $\mathcal H(\PNP{t}) \to \mathcal H(\PNP{t+\rho})$ for $t > B^\pm$, a depiction of which can be found in Figure~\ref{fig:pnp} for $M_t = \<47,49,52\>$.  Key in this endeavor is proving that for such $t$, the trade $\hh$ is irreducible (Proposition~\ref{p:hirreducible}), and no elements of $\mathcal H(\PNP{t})$ lie outside an $\hh$-neightborhood of the coordinate hyperplanes bordering $\PNP{t}$ (Proposition~\ref{p:pnpreducible}).

\begin{prop}\label{p:pnplengths}
	Fix $t > B^{\pm}$ and $\vv = (v_0,-v_1,v_2) \in \PNP{t}$.  
	\begin{enumerate}[a)]
		\item If $v_0 \le b$, then $\ell(\vv) \le 0$.
		\item If $v_2 \le a$, then $\ell(\vv) \ge 0$.
	\end{enumerate}
	In particular, $\vv = \hh$ if and only if $v_0 \le b$ and $v_2 \le a$.
\end{prop}

\begin{proof}
	Supposing $v_0 \le b$, we can use equation~(\ref{eqn:0}) to obtain
	\[t\ell(\vv) 
	= d(av_0 - bv_2) 
	< dav_0 
	\le dab 
	< t,\]
	from which $t$ can be cancelled to obtain $\ell(\vv) < 1$.  Integrality implies $\ell(\vv)\le 0$.
	Similarly, assuming $v_2 \le a$, we see
	\[t\ell(\vv) = d(av_0 - bv_2) > -dbv_2 \ge -dab >  -t.\]
	Cancelling $t$, integrality again yields the desired conclusion.  The final claim now follows from Lemma~\ref{l:basiclengthfacts}, as $\hh$ is the only homogeneous trade in $\PNP{t}$ whose first and third coordinates satsify the hypotheses of both parts.  
\end{proof}

\begin{prop}\label{p:hirreducible}
	If $t > B^{\pm}$, then $\hh$ is irreducible in $\PNP{t}$, and if $t = B^\pm$, then $\hh$ is reducible in $\PNP{t}$.  In particular, $B^\pm$ is the optimal lower bound.
\end{prop}

\begin{proof}
	Fix $t > B^\pm \ge dab$, and suppose $\uu = (u_0,-u_1,u_2)$, $\ww = (w_0,-w_1,w_2) \in \PNP{t}$ with $\hh = \uu + \ww$.  
	Since $b = h_0 = u_0 + w_0 \ge u_0$, and by a similar argument $a \ge u_2$, Proposition~\ref{p:pnplengths} implies $\uu = \hh$.  This proves the first claim.  
	
	When $t = dab$ on the other hand, observe that 
	$$\pi_t((b, -(b-1),0)) = b(t - da) - (b-1)t = t - dab = 0$$
	and analogously $\pi_t(0,-(a+1), a) = 0$,
	meaning 
	$$\hh = (b, -(b-1),0) + (0,-(a+1), a)$$
	is reducible.  
\end{proof}

\begin{prop}\label{p:pnpreducible}
	Any $(v_0,-v_1,v_2) \in \PNP{t}\setminus\{\hh\}$ with $v_0 \ge b$ and $v_2 \ge a$ is reducible.
\end{prop}

\begin{proof}
	Fix $\vv = (v_0,-v_1,v_2) \in \PNP{t}$ with $v_0 \ge b$ and $v_2 \ge a$. We claim that $v_1 \ge (a+b)$. Using $\pi_t(\vv) = 0$, we see
	\[tv_1 
	= (t-da)v_0 + (t+db)v_2 
	\ge  (t-da)b + (t+db)a 
	= t(a+b),\]
	which, after cancellation, proves the claim.  We now have $v_0 \ge b, v_1 \ge a+b$ and $v_2 \ge b$. Since $\vv \neq \hh$ by assumption, $\vv$ is a conformal sum of $\hh$ and $\vv - \hh \in \PNP{t}$, as desired.
\end{proof}

The previous proposition establishes that, for sufficiently large $t$, any elements of $\mathcal H(\PNP{t})$ necessarily lie within an $\hh$-neighborhood of the bounding coordinate planes.  With this in mind, let
\[
\SSpm{1,2}(t) = \{\vv \in \PNP{t} \mid v_0 \le b\},
\qquad
\SSpm{0,1}(t) = \{\vv \in \PNP{t} \mid v_2 \le a\},
\]
so that $\PNPhat{t} = \SSpm{1,2}(t) \cup \SSpm{0,1}(t)$ is the region represented in Figure~\ref{fig:pnp} by thin blue strips.

\begin{figure}[t]
\begin{center}
\includegraphics[height=1.8in]{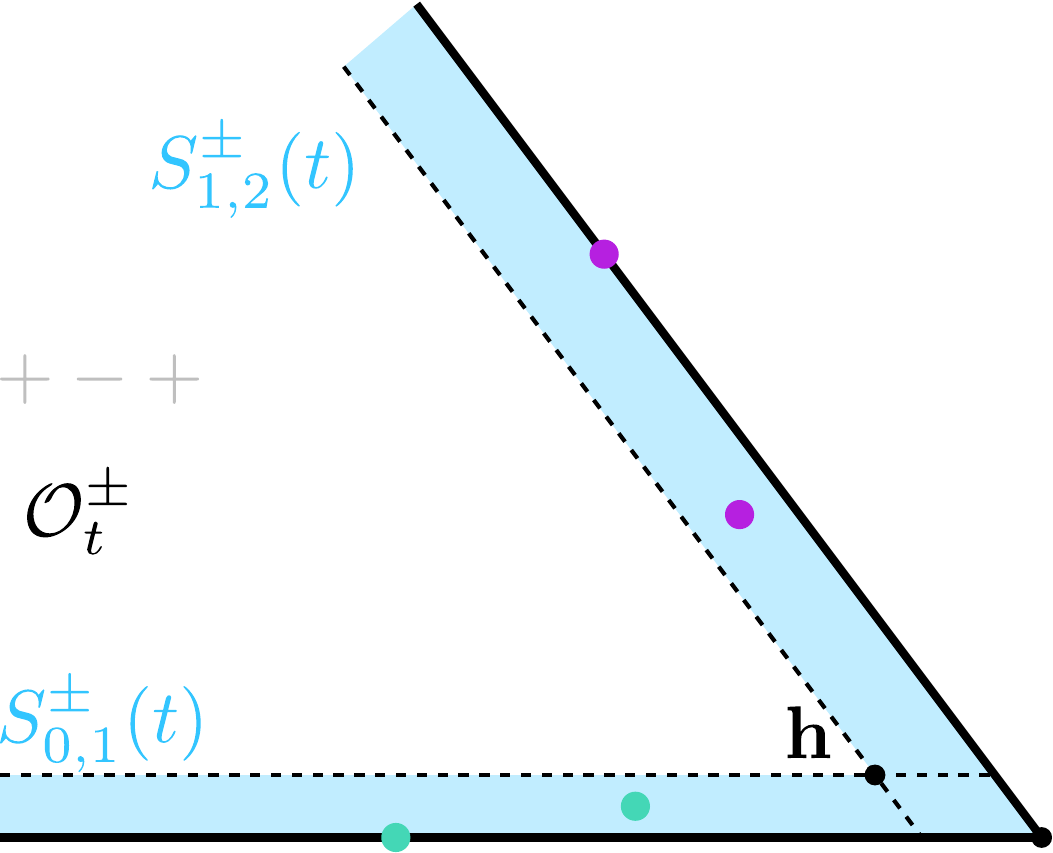}
\hspace{3em}
\includegraphics[height=1.8in]{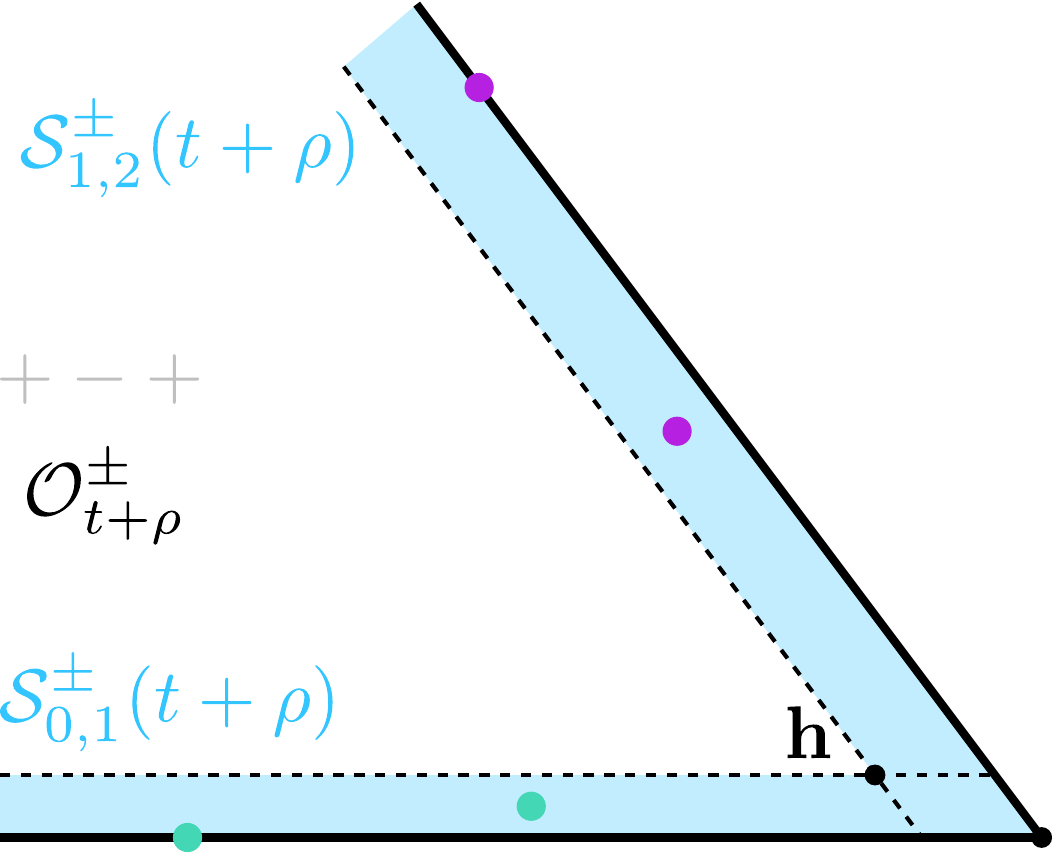}
\end{center}
\caption{The elements of $\mathcal H(\PNP{t})$ (left) and $\mathcal H(\PNP{t+\rho})$ (right) for $M_t = \<47,49,52\>$.}
\label{fig:pnp}
\end{figure}

\begin{prop}\label{p:pnpmap}
	If $t > B^{\pm}$, then the map 
	$\widehat{\varphi}_{\pm}: \PNPhat{t} \to \PNPhat{t+\rho}$ given by
	\[\widehat{\varphi}_{\pm}(\vv) = \begin{cases}
		\varphi_{1,2}(\vv) & v_0 \le b \\
		\varphi_{0,1}(\vv) & v_2 \le a
	\end{cases}\]
	is a well-defined bijection.
\end{prop}

\begin{proof}
		We begin by showing $\widehat{\varphi}_{\pm}$ is well-defined. Consider $\vv = (v_0,-v_1,v_2) \in \PNP{t}$. By Proposition~\ref{p:pnplengths}, if $v_0 \le b$ and $v_2 > a$, then $\ell(\vv) \le 0$. The application of $\varphi_{1,2}$ increases distance from each coordinate axis, ensuring that $\varphi_{1,2}(\SSpm{1,2}(t)) \subseteq \PNPhat{t+\rho}$. Symmetrically, $\varphi_{0,1}(\SSpm{0,1}(t)) \subseteq \PNPhat{t+\rho}$. It remains to show that the pre-image of any trade in the co-domain contains exactly one trade in the domain.
		By Lemma~\ref{l:basiclengthfacts}, if $v_0 \le b$ and $v_2 \le a$, then $\vv = \hh$. Since each $\varphi_{i,j}$ preserves length, we necessarily have $\widehat{\varphi}_{\pm}(\hh) = \hh$, so $\widehat{\varphi}_{\pm}$ is well-defined.
		
		Injectivity follows quickly from the piecewise linear map's trivial kernel. To prove surjectivity, we must show that $\widehat{\varphi}_{\pm}^{-1}(\PNPhat{t+\rho}) \subseteq \PNPhat{t}$. We consider $3$ cases, throughout which we denote $\ww = (w_0,-w_1,w_2) \in \PNP{t+\rho}$.
		
		First, assume $w_0 \le b$ and $w_2 \le a$. By Lemma~\ref{l:basiclengthfacts}, necessarily $\ww = \hh$. As mentioned above,  $\widehat{\varphi}_{\pm}(\hh) = \hh$, and we are done.
		
		Next, suppose $w_0 \le b$ and $w_2 > a$. $\vv = \ww - a(a+b)\ell(\ww)(\ee_1 - \ee_2)$, denoted by $\vv = (v_0, -v_1, v_2)$. Clearly $\varphi_{1,2}(\vv) = \ww$, yet we must ensure that $\vv \in \PNP{t}$. We know by Proposition~\ref{p:pnplengths} that $\ell(\ww) = -\abs{\ell(\ww)}$. Therefore, it is sufficient to show that
		\[
		v_1 = w_1 - a(a+b)\abs{\ell(\ww)} \ge 0 \qquad \text{ and } \qquad v_2 = w_2 - a(a+b)\abs{\ell(\ww)} \ge 0.
		\]
		Using equation~\eqref{eqn:2}, we observe that 
		\[dbw_1 - \rho\abs{\ell(\ww)} 
		= db(w_1 - a(a+b)\abs{\ell(\ww)}) 
		= (t+db)\abs{\ell(\ww)} + d(a+b)w_0 \ge 0
		\]
		which verifies the former inequality.
		Similarly, using equation~\eqref{eqn:1}, 
		\[
		dbw_2 - \rho\abs{\ell(\ww)}
		= db(w_2 - a(a+b)\abs{\ell(\ww)}) 
		= daw_0 + t\abs{\ell(\ww)} \ge 0,
		\]
		which verifies the latter inequality. Indeed, $\varphi_{1,2}^{-1}(\PNPhat{t+\rho}) \subseteq \PNPhat{t}$.
		
	 	Finally, assume $w_0 > b$ and $w_2 \le a$. Choose $\vv = \ww - b(a+b)\ell(\ww)(\ee_0 - \ee_1)$. It is clear that $\varphi_{0,1}(\vv) = \ww$. Since $\ell(\ww) \ge 0$, to verify $\varphi_{0,1}^{-1}(\PNPhat{t+\rho}) \subseteq \PNPhat{t}$, we need only show
		\[
		v_0 = w_0 - b(a+b)\ell(\ww) \qquad
		 \text{ and } \qquad  
		v_1 = w_1 - b(a+b)\ell(\ww).\]
		
		We arrive at these inequalities using nearly identical algebra to the previous case. The details are omitted, but we use equations~\eqref{eqn:0} and~\eqref{eqn:2} for former and latter, respectively. We conclude that $\varphi_{0,1}^{-1}(\PNPhat{t+\rho}) \subseteq \PNPhat{t}$, completing the proof.
\end{proof}

\begin{thm}\label{t:pnpgraver}
	For all $t> B^{\pm}$ and $\vv \in \PNPhat{t}$, we have $\widehat{\varphi}_{\pm}(\vv) \in \PNP{t+\rho}$ is reducible if and only if $\vv \in \PNP{t}$ is reducible.  In particular,
	\[
	\widehat{\varphi}_{\pm}(\mathcal{H}(\PNP{t})) = \mathcal{H}(\PNP{t+\rho})
	\qquad \text{and} \qquad
	\abs{\mathcal{H}(\PNP{t})} = \abs{\mathcal{H}(\PNP{t+\rho})}.
	\]
\end{thm}

\begin{proof}
	Fix $\uu, \vv, \ww \in \PNPhat{t}$ and suppose $\vv = \uu + \ww$.  Since this sum is conformal, if $\vv \in \SSpm{1,2}(t)$, then $a > v_0 = u_0 + w_0$, while if $\vv \in \SSpm{0,1}(t)$, then $b > v_2 = u_2 + w_2$.  As~such, we can assume either $\uu, \vv, \ww \in \SSpm{0,1}(t)$ or $\uu, \vv, \ww \in \SSpm{1,2}(t)$.  The piecewise linearity of $\widehat{\varphi}_{\pm}$ thus ensures $\widehat{\varphi}_{\pm}(\ww) = \widehat{\varphi}_{\pm}(\uu) + \widehat{\varphi}_{\pm}(\vv) \in \PNPhat{t+\rho}$ is reducible.  Bijectivity from Proposition~\ref{p:pnpmap} ensures the converse holds as well.  
	
	Lastly, by Proposition ~\ref{p:pnpreducible} any trade in $\mathcal H(\PNP{t})$ must lie $\SSpm{0,1}(t)$ or $\SSpm{1,2}(t)$.  As such, the above argument implies $\widehat{\varphi}_{\pm}$ restricts to a bijection $\mathcal H(\PNPhat{t}) \to \mathcal H(\PNPhat{t+\rho})$, as desired.  
\end{proof}

\section{The Hilbert Basis of $\PPN{t}$}
\label{sec:ppn}

We now turn our attention to the orthant $\PPN{t}$.  As in the previous section, we begin by characterizing the lengths of trades (Proposition~\ref{p:ppnlengths}), identifying where in $\PPN{t}$ the trades in $\mathcal H(\PPN{t})$ must lie (Theorem~\ref{t:ppnreducible}), and constructing bijective maps (Proposition~\ref{p:ppnmaps}) used to obtain $\mathcal H(\PPN{t+\rho})$ from $\mathcal H(\PPN{t})$ in Theorem~\ref{t:ppngraver}.  

\begin{prop}\label{p:ppnlengths}
	Any $\vv = (v_0,v_1,-v_2) \in \PPN{t}$, satisfies $\ell(\vv) > 0$. Consequently, any length~$d$ trade is irreducible, and if $\vv \in \PPN{t}$ is reducible, then $\ell(\vv) \ge 2d$.  
\end{prop}

\begin{proof}
	We have $k\hh \in \pm \PNP{t}$ for all $k \in \ZZ$, so $\ell(\vv) \neq 0$. Using equation~\eqref{eqn:2},
	\[(t+db)\ell(\vv) = d((a+b)v_0 + bv_1) > 0,\]
	and since $t + db > 0$, it follows that $\ell(\vv) > 0$. The subsequent claims then follow from the linearity of $\ell$.
\end{proof}


	
\begin{prop}\label{p:lendexists}
	For $t > B^+$, there exists $\vv \in \PPN{t}$ with $\ell(\vv) = d$, while if $t = B^+$, then every $\vv \in \PPN{t}$ has $\ell(\vv) > d$.  
	In particular, $B^+$  is the optimal lower bound.
\end{prop}

\begin{proof}
	Consider the 2-generated semigroup
	\[S = \<b,a+b\>.\]
	It is known that $\mathcal{F}(S) = b(a+b) - 2b - a = (b-1)(a+b) - b$. As such, if $$t > \mathcal{F}(S) - db = (b-1)(a+b) - b - bd,$$
	it follows there exists a positive integral solution $(v_0, v_1)$ to the Diophantine equation
	\begin{equation}\label{eq:PPNsubsemigroup}
		bv_1 + (a+b)v_0 = t+db.
	\end{equation}
	Fixing this choice of $v_1$ and $v_0$, when $t = da$, we have $d(a+b) = bv_1 + (a+b)v_0 \in S$, which means by definition that $(v_1,v_0)$ is a factorization of $d(a+b)$ in $S$.  Clearly $(0,d)$ is another factorization of $d(a + b)$ in $S$, one supported only on the maximal generator of $S$, so it must be the minimal length factorization.  Additionally, the maximal element $n \in S$ with factorization length strictly less than $d$ is $(d-1)(a+b)$. As such, for $t > da$, since $t+db = d(a+b) > (d-1)(a+b)$, it follows that $v_0 + v_1 \ge d$. 
	
	We claim $\vv = (v_0,v_1, -(v_0 + v_1-d)) \in \PPN{t}$.  It is clear $\vv$ has the proper sign pattern by the reasoning above, so we need only verify $\pi_t(\vv) = 0$. After substituting the expression for $t + db$ given in equation~\eqref{eq:PPNsubsemigroup}, we see that
	\[\pi_t(\vv) = (t-da)v_0 + tv_1 - (t+db)(v_0 + v_1) + d(t+db) = d((a+b)v_0 + bv_1)\left(-1 + 1\right) = 0\]
	and since $\ell(\vv) = d$, we have constructed the desired trade. 

	Optimality of $B^+$ follows from the fact that any such $\vv$ yields an expression
	\[
		(a+b)v_0 + bv_1 = t+db,
	\]
	which is impossible when $t+db = \mathcal{F}(S)$. 
\end{proof}

Proposition~\ref{p:lendexists} tells us that for $t > B^+$, there is at least one length $d$ trade in $\PPN{t}$.  Let $\al \in \PPN{t}$ denote the length $d$ trade with $\alpha_0$ minimal.  Note that $0 \le \alpha_0 < b$ since minimality requires $\al - \hh \notin \PPN{t}$.   
By the division algorithm, there exist $q, \beta_1 \in \ZZ$ with $0 \le \beta_1 < a + b$ such that
$\alpha_1 = q(a + b) + \beta_1$.  Let 
$\be = \al + q\hh.$
Note that $\be = (\beta_0, \beta_1, -\beta_2) \in \PPN{t}$ since 
$\beta_0 = \alpha_0 + qb \ge 0$ and equation~\eqref{eqn:0} implies
\[da\beta_1 = d(a+b)\beta_2 - d(t-da) < d(a+b)\beta_2,\]
so $\beta_2 \ge 0$.  
We have thus located length $d$ trades
\[
\al = (\alpha_0, \alpha_1, -\alpha_2) \in \PPN{t}
\text{ with }
\alpha_0 < b
\quad \text{and} \quad
\be = (\beta_0, \beta_1, -\beta_2) \in \PPN{t}
\text{ with }
\beta_1 < a + b,
\]
so that, in particular, $\al$ has minimal first coordinate and $\be$ has minimal second coordinate among trades in $\PPN{t}$ with length $d$.  These form the endpoints of the line segment that dilates as $t$ increases, as depicted in Figure~\ref{fig:ppn}.  Note that $\al$ and $\be$ may coincide.  

We will see in Proposition~\ref{p:pnpreducible} that it is in
\[
\SSplus{1,2}(t) = \{(v_0,v_1,-v_2) \in \PPN{t} \mid v_0 < b\}
\quad \text{and} \quad 
\SSplus{0,2}(t) = \{(v_0,v_1,-v_2) \in \PPN{t} \mid v_1 < a+b\}
\]
where irreducible trades of coordinate sum greater than $d$ must lie.  
For $t > B^+$, $\al \in \SSplus{1,2}(t)$ and $\be \in \SSplus{0,2}(t)$ by Theorem~\ref{p:lendexists}. Moreover, it is worth noting that the intersection $\SSplus{1,2}(t) \cap \SSplus{0,1}(t)$ can be non-empty only when $\al = \be$, which, if it occurs for~$M_t$, does not occur for $M_{t+k\rho}$ for any positive $k$.   

\begin{figure}[t]
\begin{center}
\includegraphics[height=1.8in]{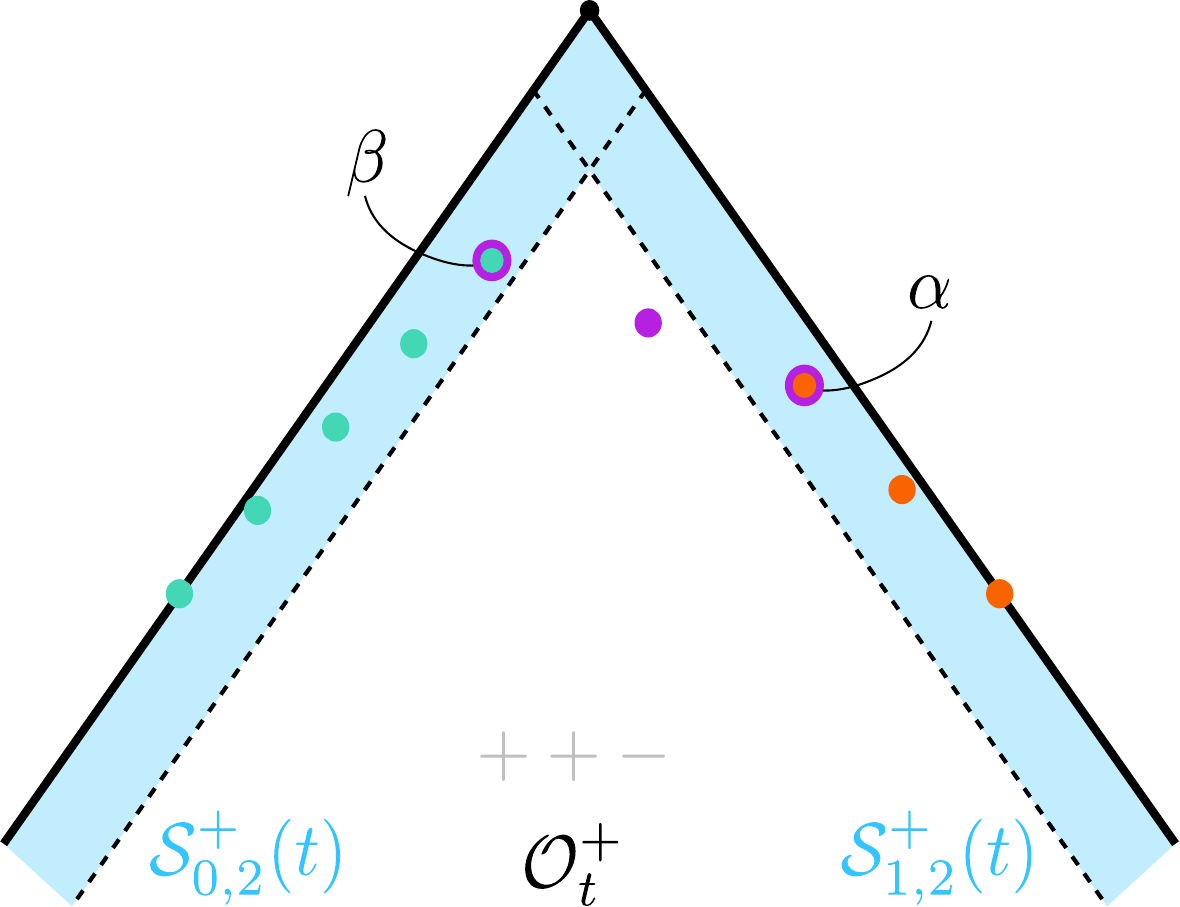}
\hspace{3em}
\includegraphics[height=1.8in]{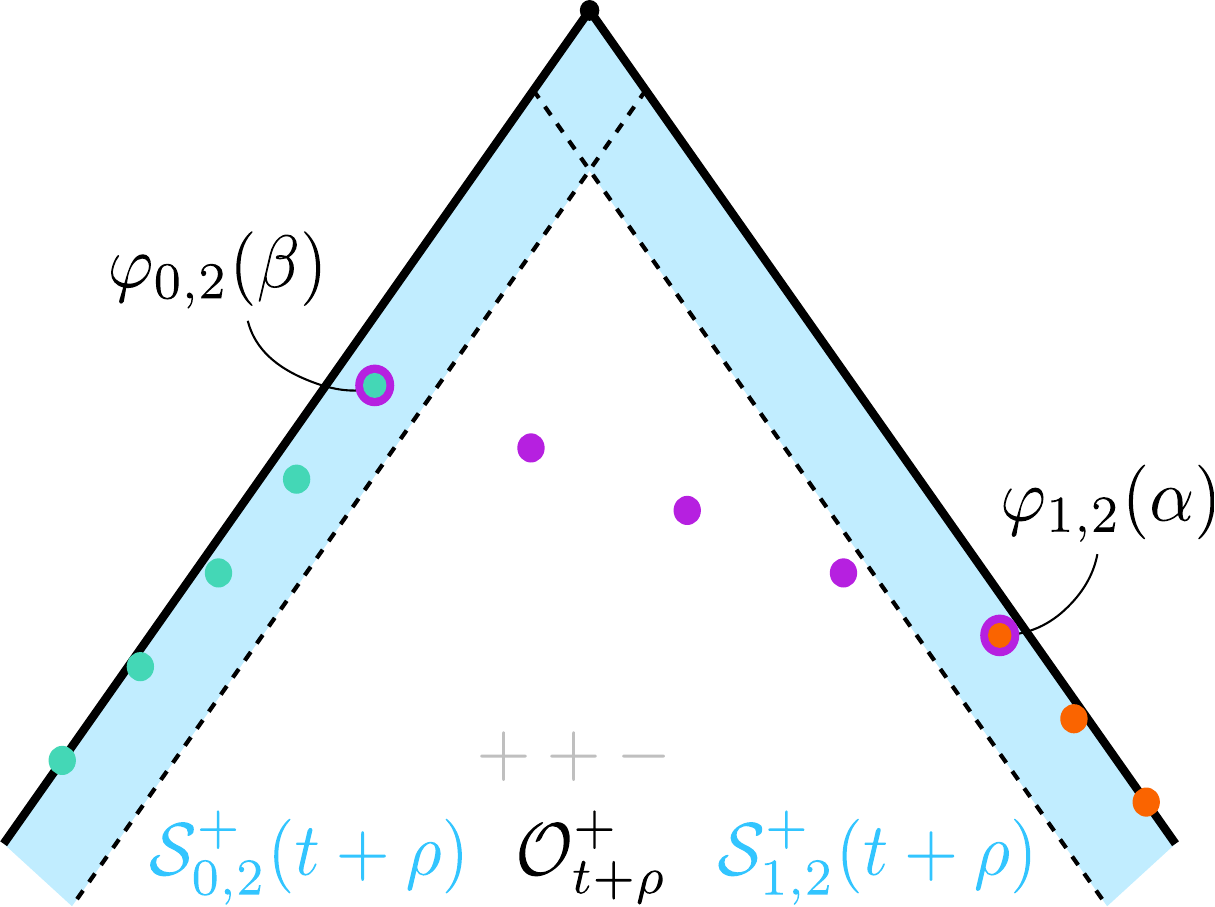}
\end{center}
\caption{The elements of $\mathcal H(\PPN{t})$ (left) and $\mathcal H(\PPN{t+\rho})$ (right) for $M_t = \<47,49,52\>$.}
\label{fig:ppn}
\end{figure}

\begin{thm}\label{t:ppnreducible}
	For $t > B^+$, if $\vv = (v_0,v_1,-v_2)\in \PPN{t}\setminus\left(\SSplus{1,2}(t) \cup \SSplus{0,2}(t)\right)$ and  $\ell(\vv) > d$, then $\vv$ is reducible.
\end{thm}

\begin{proof}
	For $t > B^+$, we are guaranteed the existence of $\al$ and $\be$ as described above, each with coordinate sum $d$, satisfying $\alpha_0 < b$ and $\beta_1 < a+b$. We will first attempt to write $\vv = (\vv-\al) + \al$. If $\vv - \al \in \PPN{t}$, then $\vv$ is clearly reducible in $\PPN{t}$, and we are done. By assumption that $\alpha_0 < b \le v_0$, which implies $v_0 - \alpha_0 > 0$. Therefore, we must have $v_1 - \alpha_1 < 0$ or $v_2 - \alpha_2 < 0$. We claim that if $v_1 - \alpha_1 > 0$, it follows that $v_2 - \alpha_2 > 0$. Indeed, if we express $\alpha_1$ and $\alpha_2$ in terms of $\alpha_0$ using our linear constraints, $\pi(\al) = 0$ and $\ell(\al) = d$, we obtain the system of equations
	\[\alpha_1 - \alpha_2 
	= d - \alpha_0 \qquad a\alpha_1 - (a+b)\alpha_2 
	= -(t-da)\]
	
	We can solve this for $\alpha_2$ in terms of $\alpha_0$ to obtain the equation
	\begin{equation}\label{eq:reducibilityrelation}
		b\alpha_2 = t - a\alpha_0.
	\end{equation}
	Using $\pi_t(\vv) = 0$, $\ell(\al) = d$, and $\ell(\vv) = nd$ for some $n \ge 2$, we obtain the expression.
	\[(a+b)v_2 
	= (t-da)n + av_1 
	\ge (t-da)n + a\alpha_1 
	\ge (t-da) + a(d - \alpha_0 + \alpha_2).\]
	This chain of inequality can be extended using some more algebra, which gives
	\[(a+b)v_2 
	\ge t - a\alpha_0 + a\alpha_2 
	= (a+b)\alpha_2 
	\ge (a+b)\alpha_2,\]
	and we can cancel $a+b$ from the left and rightmost expressions in the previous centered chain to see that $v_2 \ge \alpha_2$, as claimed. We therefore conclude that $\vv - \al \not\in \PPN{t}$ implies $v_1 - \alpha_1 < 0$, the only remaining possibility. 
	
	Assuming this is true, we must locate a non-negative integer $k \le q$, where $\be = \al + q\hh$, such that $\vv - (\al + k\hh) \in \PPN{t}$ in order to write 
	\[\vv = (\vv - (\al + k\hh)) + (\al + k\hh).\] 
	By assumption, we have $v_1 < \alpha_1$, $a+b \le v_1$, $\alpha_1 = \beta_1 - qh_1$, and $\beta_1 < a + b$. We can combine these inequalities to obtain
	\[v_1 - \alpha_1 
	\ge a+b - \alpha_1 
	= a+b - \beta_1 + qh_1 
	> qh_1,\]
	which implies that $v_1 - (\alpha_1 + qh_1) > 0$. Since $v_1 - \alpha_1 < 0$, we are guaranteed the existence of a minimal positive $k$ such that $v_1 \ge \alpha_1 + kh_1 = \alpha_1 - k(a+b)$. 
	
	We claim it follows that $v_2 \ge \alpha_2 + kh_2 = \alpha_2 + ka$. To see this, we use the mentioned assumptions along with $\ell(\al) = d$, $\ell(\vv) = nd$, $n\ge 2$ to obtain the chain of inequality
	\[(a+b)v_2 
	= (t-da)n + av_1 
	\ge (t-da)n + a(\alpha_1 - k(a+b)) 
	\ge 2(t-da) + a(d - \alpha_0 + \alpha_2) + ka(a+b).\]
	Using equation $3.1$, we can continue the chain with similar algebra to previous cases, which yields
	\[(a+b)v_2 
	\ge (t-da) + (t-a\alpha_0) + a\alpha_2 +ka(a+b) 
	= (t-da) + (a+b)(\alpha_2 + ka) 
	\ge (a+b)(\alpha_2 + ka),\]
	which gives the desired inequality after cancelling $(a+b)$ from the left and right-most expressions in the previous chain. It only remains to show that $v_0 - (\alpha_0 + kb)  \ge 0$.
	
	Since $k$ is the minimal integer such that $v_1 - (\alpha_1 - k(a+b)) \ge 0 $, it follows that $v_1 - (\alpha_1 - k(a+b)) < a+b$, or equivalently,
	\[v_1 \le \alpha_1 - (k -1)(a+b) - 1.\]
	For notational convenience, we will define $Q = (k-1)(a+b) + 1$. Using $\ell(\al) = d$, $\ell(\vv) = nd$ for some integer $n\ge 2$, $v_1 \le \alpha_1 - Q$ and $\pi_t(\vv) = $, we have
	\[(a+b)v_0 
	= (t+db)n - bv_1 
	\ge 2(t+db) -b(\alpha_1 - Q) 
	= 2(t+db) -b(d - \alpha_0 + \alpha_2)+bQ.\]
	Using Equation ~\ref{eq:reducibilityrelation} and back-substitution for $Q$, we use algebra to continue the chain, obtaining
	\[(a+b)v_0 
	\ge t+db + (a+b)(\alpha_0 + b(k-1)) + b(d+1).\]
	By applying our bound for $t$, we obtain
	\[(a+b)v_0 
	> (a+b)(\alpha_0 + kb -1),\]
	and after cancellation we have $v_0  - (\alpha_0 + kb) > -1 \ge 0$, as desired. This completes the proof.
\end{proof}

Theorem~\ref{t:ppnreducible} tells us that for $t > B^+$, any irreducible trade $\vv \in \PPN{t}$ must satisfy $v_0 < b$, $v_1 < a+b$ or $\ell(\vv) = d$.  
By Proposition~\ref{p:ppnlengths}, if $\ell(\vv) = d$ then $\vv$ is irreducible.  All other irreducible trade must lie $\PPNhat{t}$. 

\begin{prop}\label{p:ppnmaps}
	For $t > B^+$, the maps 
	\[
	\begin{array}{r@{}c@{}l@{\qquad}l@{\qquad}r@{}c@{}l}
		\widehat{\varphi}_{1,2}^+: \SSplus{1,2}(t) &{}\longrightarrow{}& \SSplus{1,2}(t+\rho)
		& \text{and} &
		\widehat{\varphi}_{0,2}^+: \SSplus{0,2}(t) &{}\longrightarrow{}& \SSplus{0,2}(t+\rho)
		\\
		\vv &\longmapsto& \varphi_{1,2}(\vv) 
		& &
		\vv &\longmapsto& \varphi_{0,2}(\vv)
	\end{array}
	\]
	are bijections.  
\end{prop}

\begin{proof}
	Both $\widehat\varphi_{1,2}^+$ and $\widehat\varphi_{0,2}^+$ are well-defined as the they leave untouched the first and second coordinates, respectively.  
	We have seen in Theorem~\ref{t:big_phi_properties} that each $\varphi_{i,j}$ is bijective, so injectivity of $\widehat\varphi_{1,2}^+$ and $\widehat\varphi_{0,2}^+$ is immediate.  As such, we need only verify surjectivity.  
	
	Fix $\ww \in \SSplus{1,2}(t+\rho)$.  We claim $\vv = (v_0, v_1, -v_2) = \ww - a(a+b)\ell(\ww)(\ee_1 - \ee_2) \in \SSplus{1,2}(t)$.  It is clear $\varphi_{1,2}(\vv) = \ww$, and $v_0 = w_0$ by definition, so we need only to ensure $v_1$ and $v_2$ have appropriate signs, i.e., that  $w_1, w_2 \ge a(a+b)\ell(\ww)$.
	
	To show $w_2 \ge a(a+b)\ell(\ww)$, we begin with representation ~(\ref{eqn:1}) of $\pi_{t+\rho}(\ww) = 0$. Since $w_0$ is an integer, $w \in \SSplus{1,2}$ implies $w_0 \le b - 1$. Using algebra and this bound, we obtain the inequality
	\[dbw_2 - \rho\ell(\ww) 
	= t\ell(\ww) - daw_0 
	\ge t\ell(\ww) - da(b-1).\]
	The left hand side factors to $db(w_2 - a(a+b)\ell(\ww))$ upon substituting the value of~$\rho$. By Lemma~\ref{l:basiclengthfacts}, $d\mid \ell(\ww)$, so we write $\ell(\ww) =nd$ for some positive integer $n$ by Proposition~\ref{p:ppnlengths}.  Using these substitutions, the previous centered inequality becomes
	\begin{equation}\label{eqn:branchpoint}
		b(w_2 - a(a+b)\ell(\ww)) \ge (nt - a(b-1)) \ge t - a(b-1).
	\end{equation}
	If $b - 1 \le d$, we the right hand side of~\eqref{eqn:branchpoint} is nonnegative since $t-da \ge 0$. Otherwise, $b-1 \ge d+1$, and we use $t > B+ = (b-1)(a+b) - (d+1)b$, applied to~\eqref{eqn:branchpoint}, obtaining
	\begin{align*}
		b(w_2 - a(a+b)\ell(\ww)) 
		&> (b-1)(a+b)-(d+1)b - a(b-1) = b(b-1) - (d+1)b \\
		&\ge b(d+1) - (d+1)b = 0,
	\end{align*}
	as desired.

	We now verify that $w_1 \ge a(a+b)\ell(\ww)$, we begin with representation ~(\ref{eqn:2}) of $\pi_{t+\rho}(\ww)$, using algebra along with Lemma~\ref{l:basiclengthfacts} to rewrite 
	\[
	dbw_1 - \rho\ell(\ww) 
	= (t+db)\ell(\ww) - d(a+b)w_0 
	\ge d\big[n(t+db) - \big((a+b)(b-1)\big)\big].
	\]
	If $\ell(\ww) = d$, we are ensured that $\ww = \varphi_{1,2}(\al)$, so we may assume $n \ge 2$. Factoring the left hand side, we consolidate the inequality, which yields 
	\begin{align*}
		b(w_1 - a(a+b)\ell(\ww)) 
		&\ge 2(t+db) - \big((a+b)(b-1)\big) \\
		&\ge 2\big((b-1)(a+b) - b+1\big) - (a+b)(b-1) \\
		&= (b-1)(a+b)- 2(b-1) = (b-1)(a+b-2) \ge 0,
	\end{align*}
	since $a+b \ge 2$ and $b \ge 1$. This completes the proof of surjectivity of $\widehat\varphi_{1,2}^+$.
	
 	To see that $\widehat{\varphi}_{0,2}^+$ is surjective, we similarly must verify that $w_0 \ge ab\ell(\ww)$ and that $w_2 \ge ab\ell(\ww)$. This can be easily checked using analagous algebra, using~\eqref{eqn:2} and~\eqref{eqn:0} and the bound $t + db \ge (b-1)(a+b - 1)$, completing the proof.
\end{proof}

Note that, in contrast to Section~\ref{sec:pnp}, the maps defined in Proposition~\ref{p:ppnmaps} need not agree on the intersections of their domains.  Indeed, when the intersection is nonempty, $\al = \be$ is the only trade residing therein, and $\widehat \varphi_{1,2}^+$ and $\widehat \varphi_{0,2}^+$ send $\al$ and $\be$ to distinct trades in $\PPN{t+\rho}$.

\begin{thm}\label{t:ppngraver}
	For $t > B^+$ and $\vv \in \SSplus{1,2}(t)$, the trade $\widehat{\varphi}_{1,2}^+(\vv) \in \SSplus{1,2}(t+\rho)$ is reducible if and only if $\vv \in \PPN{t}$ is reducible.  The map $\widehat \varphi_{0,2}^+$ likewise preserves reducibility.  Thus,
	\[
		\mathcal{H}(\PPN{t+\rho}) =  \bigcup_{j\in\{0,1\}}\widehat{\varphi}_{j,2}^+\left(\mathcal{H}(\PPN{t}) \cap \SSplus{j,2}(t)\right) \cup \{\varphi_{1,2}(\al), \varphi_{1,2}(\al) + \hh, \ldots, \varphi_{0,2}(\be))\}
	\]
	and in particular $|\mathcal{H}(\PPN{t+\rho})| = |\mathcal{H}(\PPN{t})| + da$.
\end{thm}

\begin{proof}
	The bijections $\widehat \varphi_{1,2}^+$ and $\widehat \varphi_{0,2}^+$ established in Proposition~\ref{p:ppnmaps} preserve irreducibility by a similar argument to the one given in Theorem~\ref{t:pnpgraver}, and thus restrict to bijections $\mathcal{H}(\PPN{t}) \cap \SSplus{1,2}(t) \to \mathcal{H}(\PPN{t+\rho}) \cap \SSplus{1,2}(t+\rho)$ and $\mathcal{H}(\PPN{t}) \cap \SSplus{1,2}(t) \to \mathcal{H}(\PPN{t+\rho}) \cap \SSplus{1,2}(t+\rho)$, respectively. 
	We now need only consider the irreducible trades that lie outside $\PPNhat{t}$.  By~Proposition~\ref{p:ppnlengths}, every such trade has length $d$.  Since $\varphi_{1,2}$ and $\varphi_{0,2}$ preserve length and fix extremal entries, we claim
	\[
	da = \big\vert\{\varphi_{1,2}(\al), \varphi_{1,2}(\al) +\hh,\ldots,\varphi_{1,2}(\be) - \hh,\varphi_{1,2}(\be)\}\big\vert - \big\vert\{\al, \al +\hh,\ldots,\be - \hh,\be\}\big\vert
	\]
	By construction, $\al$ and $\be$ are the extremal length $d$ trades, and $\be - \al = k\hh$ for some $k \in \ZZ_{\ge 0}$.  
	To complete the proof, we need only observe 
	\[
	\varphi_{0,2}(\be) - \varphi_{1,2}(\al) = \be - \al + dab(\ee_0 - \ee_2) - da(a+b)(\ee_1 - \ee_2) =(k + da)\hh.
	\]
	All together, we obtain $|\mathcal{H}(\PPN{t+\rho})| - |\mathcal{H}(\PPN{t})| = (k + da - k) = da$, as desired.
\end{proof}

\section{Constructing the Graver basis}
\label{sec:graverbasis}

The results in this previous section naturally extend via a symmetric argument to the following result for the orthant $\NPP{t}$.  We record the results here, omitting the proof.

\begin{thm}\label{t:nppgraver}
	Let
	\[
	\SSminus{0,2}(t) = \{(-v_0, v_1, v_2) \in \NPP{t} \mid v_2 < a\}
	\quad \text{and} \quad
	\SSminus{0,1}(t) = \{(-v_0, v_1, v_2)  \in \NPP{t} \mid v_1 < a+b\}.
	\]
	For any $t > B^-$, the maps 
	\[
	\begin{array}{r@{}c@{}l@{\qquad}l@{\qquad}r@{}c@{}l}
		\widehat{\varphi}_{0,2}^-: \SSminus{0,2}(t) &{}\longrightarrow{}& \SSminus{0,2}(t+\rho)
		& \text{and} &
		\widehat{\varphi}_{0,1}^-: \SSminus{0,1}(t) &{}\longrightarrow{}& \SSminus{0,1}(t+\rho)
		\\
		\vv &\longmapsto& \varphi_{0,2}(\vv) 
		& &
		\vv &\longmapsto& \varphi_{0,1}(\vv)
	\end{array}
	\]
	are well-defined bijections that preserve irreducibility. 
	Furthermore, there exist trades $\boldsymbol{\xi} \in \SSminus{0,1}(t)$ and $\boldsymbol{\omega} \in \SSminus{0,2}(t)$ with $\xi_2 < a$ and $\omega_1 < a + b$ minimal among those with coordinate sum $\ell(\boldsymbol{\xi}) = \ell(\boldsymbol{\omega}) = -d$.
	As such,
	\[
	\mathcal{H}(\NPP{t+\rho}) =  \bigcup_{j\in\{1,2\}}\widehat{\varphi}_{0,j}^-\left(\mathcal{H}(\NPP{t}) \cap \SSminus{0,j}(t)\right)
	\cup \{\varphi_{0,1}(\boldsymbol{\xi}), \varphi_{0,1}(\boldsymbol{\xi}) + \hh, \ldots, \varphi_{0,2}(\boldsymbol{\omega}))\}
	\]
and in particular,
	\[\abs{\mathcal{H}(\NPP{t+\rho})} = \abs{\mathcal{H}(\NPP{t})} + db.\]
\end{thm}

We are now ready to prove Theorem~\ref{t:graversize}.  

\begin{proof}[Proof of Theorem~\ref{t:graversize}]
	By definition, $Gr(M_{t+\rho})$ is the union of the Hilbert bases of $\PPN{t}$, $\PNP{t}$, and $\NPP{t}$ and their negations, any two of which have disjoint interiors.  Their intersections are thus contained in the coordinate planes, and each contain at most one of the one trades with a zero entry.  
	As such, 
	\[
		|Gr(M_t)| = 2\bigg[|\mathcal H(\PNP{t})| + |\mathcal H(\PPN{t})| + |\mathcal H(\NPP{t})| - 3\bigg]
	\]
	and thus for $t > B$, 
	\[
	|Gr(M_{t+\rho})| - |Gr(M_{t})| = \sum_{s\in\{+,\pm,-\}} 2\left[|\mathcal{O}^s_{t+\rho}| - 2|\mathcal{O}^s_{t}|\right] = 2d(a+b)
	\]
	by Theorems~\ref{t:pnpgraver}, \ref{t:ppngraver}, and~\ref{t:nppgraver}, as desired.
\end{proof}


\begin{table}[t]
	\begin{center}
		$\scriptstyle
		\begin{array}{@{}c | c | c | c@{}}
			& \mathcal S_{1,2}(19) & \mathcal S_{0,2}(19) & \mathcal S_{0,1}(19) \\
			\hline
			\mathcal H(\PPN{19}) &
			\begin{blockarray}{rrr}
				&& \color{blue}{\al} \\
				\begin{block}{(rrr)}
					0 & 1 & \color{blue}{\phantom{-0}2} \\
					22 & 13 & \color{blue}{4}\\
					-19 & -12 & \color{blue}{-5} \\
				\end{block}
			\end{blockarray}
			&
			\begin{blockarray}{rrrrr}
				\color{blue}{\be} &&&& \\
				\begin{block}{(rrrrr)}
					\color{blue}{\phantom{-0}2} & \phantom{-0}7 & 12 & 17 & 22 \\
					\color{blue}{4} &3 & 2 & 1 & 0 \\
					\color{blue}{-5} & -8 & -11 & -14 & -17 \\
				\end{block}
			\end{blockarray}
			& \mathbf{0}
			\\[-1em]
			\mathcal H (\PNP{19}) & 
			\begin{blockarray}{(rr)}
				0 & 1 \\
				-22 & -9 \\
				19 & 7 \\
			\end{blockarray}
			& 
			\begin{blockarray}{r}
				\hh \\
				\begin{block}{(r)}
					3 \\
					-5 \\
					2 \\
				\end{block}
			\end{blockarray}
			&  
			\begin{blockarray}{(rr)}
				11 & 19 \\
				-11 & -17 \\
				1 & 0
			\end{blockarray}
			\\[-1.3em]
			\mathcal H(\NPP{19})  & 
			\mathbf{0}&
			\begin{blockarray}{rr}
				& \color{red}{\boldsymbol{\omega}} \\
				\begin{block}{(rr)}
					-22 &\color{red}{-5} \\
					0 & \color{red}{1} \\
					17 & \color{red}{3} \\
				\end{block}
			\end{blockarray}
			&
			\begin{blockarray}{rr}
				\color{red}{\boldsymbol{\xi}} & \\
				\begin{block}{(rr)}
					\color{red}{-8} & -19  \\
					\color{red}{6} & 17 \\
					\color{red}{1} & 0 \\
				\end{block}
			\end{blockarray}
		\end{array}$
	\end{center}
	
	\medskip
	\begin{center}
		$\scriptstyle
		\begin{array}{@{}c | c | c | c@{}}
			& \mathcal S_{1,2}(79) & \mathcal S_{0,2}(79) & \mathcal S_{0,1}(79) \\
			\hline
			\mathcal H(\PPN{79}) 
			&
			\begin{blockarray}{rrr}
				&& \color{blue}{\al^*} \\
				\begin{block}{(rrr)}
					0 & 1 & \color{blue}{2} \\
					82 & 53 & \color{blue}{24} \\
					-79 & -52 & \color{blue}{-25} \\
				\end{block}
			\end{blockarray}
			&
			\begin{blockarray}{rrrrr}
				\color{blue}{\be^*} &&&& \\
				\begin{block}{(rrrrr)}
					\color{blue}{14} & 31 & 48 & 65 & 82 \\
					\color{blue}{4} &3 & 2 & 1 & 0 \\
					\color{blue}{-17} &-32 & -47 & -62 & -77 \\
				\end{block}
			\end{blockarray}
			&
			\mathbf{0}\\[-1em]
			\mathcal H (\PNP{79})
			&
			\begin{blockarray}{(rr)}
				0 & 1 \\
				-82 & -29  \\
				79 & 27 
			\end{blockarray}
			& 
			\begin{blockarray}{r}
				\hh \\
				\begin{block}{(r)}
					3 \\
					-5 \\
					2 \\
				\end{block}
			\end{blockarray}
			&  
			\begin{blockarray}{(rr)}
				41 & 79 \\
				-41 & -77 \\
				1 & 0
			\end{blockarray}
			\\[-1.3em]
			\mathcal H(\NPP{79}) 
			& 
			\mathbf{0}&
			\begin{blockarray}{rr}
				& \color{red}{\boldsymbol{\omega}^*} \\
				\begin{block}{(rr)}
					-82 &\color{red}{-17} \\
					0 & \color{red}{1} \\
					77 & \color{red}{15} \\
				\end{block}
			\end{blockarray}
			&
			\begin{blockarray}{rr}
				\color{red}{\boldsymbol{\xi}^*} & \\
				\begin{block}{(rr)}
					\color{red}{-38} & -79  \\
					\color{red}{36} & 77 \\
					\color{red}{1} & 0 \\
				\end{block}
			\end{blockarray}
		\end{array}$
	\end{center}
	\caption{The trades in $Gr(M_{19})$ (top) and $Gr(M_{79})$ (bottom) within one $\hh$-heighborhood of a coordinate plane.  
	}
	\label{tb:mainexample}
\end{table}

\begin{example}\label{e:mainexample}
	Consider the semigroup $M = \<77,79,82\>$.  The middle generator gives $t = 79$, the differences between successive generators yield $\rr = (-2,0,3)$, i.e. $a = 2$ and $b = 3$, and $d = 1$ since $\gcd(a,b) = 1$, so we have $M = M_{79}$.  To ensure we may apply Theorem~\ref{t:graversize}, we verify $79 > B = \max\{B^+, B^{\pm}, B^-\} = \max\{6,9,10\} = 10$.
	After calculating $\rho = 30$, we apply the division algorithm to obtain $t = 2\rho + 19$ with $19 > B$.  
	
	To construct $Gr(M_{79})$, we begin by computing $Gr(M_{19})$.  The software package \texttt{4ti2}~\cite{4ti2} outputs the matrix 
	\[
		Gr(M_{19}) = 
		\scriptstyle
		\left(\begin{array}{rrrrrrrrrrrrr}
			-19 & 11 & -8 & \color{red}{3} & -5 & -2 & 1 & -7 & -12 & -1 & -17 & -22 & 0 \\
			17 & -11 & 6 & \color{red}{-5}& 1 & -4 & -9 & -3 & -2 & -13 & -1 & 0 & -22 \\
			0 & 1 & 1 & \color{red}{2} & 3 & 5 & 7 & 8 & 11 & 12 & 14 & 17 & 19
		\end{array}\right)
	\]
	whose columns and their negatives are the 26 trades in $Gr(M_{19})$.  Indicated in red is the trade $\hh = (3,-5,2)$.  
	Partitioning $Gr(M_{19})$ as dictated by Theorems~\ref{t:pnpgraver}, \ref{t:ppngraver}, and~\ref{t:nppgraver} yields the breakdown in Table~\ref{tb:mainexample}, we proceed to construct $Gr(M_{79})$ by applying the appropriate choice of $\varphi_{i,j}$ twice to trades in each column.

	To complete the construction, we add multiples of $\hh = (3,-5,2)$ to $\al^* = \varphi_{1,2}^2(\al)$ and $\boldsymbol{\xi}^* = \varphi_{0,1}^2(\boldsymbol{\xi})$ until arriving at $\be^* = \varphi_{0,2}^2(\be)$ and $\boldsymbol{\omega}^* = \varphi_{0,2}^2(\boldsymbol{\omega})$, respectively, obtaining the columns in the matrices
	\[
		\begin{blockarray}{rrrrr}
		\al^* &&&& \be^* \\
		\begin{block}{r(rrr)r}
		 2 & 5 & 8 & 11 & 14 \\
		 24 & 19 & 14 & 9 & 4 \\
		 -25 & -23 & -21 & -19 & -17 \\
		\end{block}
		\end{blockarray}
		\quad \text{and} \quad
		\begin{blockarray}{rrrrrrrr}
		\boldsymbol{\xi}^* &&&&&&& \boldsymbol{\omega}^* \\
		\begin{block}{r(rrrrrr)r}
		-38 & -35 & -32 & -29 & -26 & -23 & -20 & -17 \\
		36 & 31 & 26 & 21 & 16 & 11 & 6 & 1 \\
		1 & 3 & 5 & 7 & 9 & 11 & 13 & 15 \\
		\end{block}
		\end{blockarray}
	\]
	Indeed, one may verify with \texttt{4ti2} that $Gr(M_{79})$ consists of the 23 trades 
	\[\left(\begin{array}{rrrrrrrrrrrr}
		-79 & 41 & -38 & 3 & -35 & -32 & -29 & -26 & -23 & -20 & -17 & -14 \\
		77 & -41 & 36 & -5 & 31 & 26 & 21 & 16 & 11 & 6 & 1 & -4 \\
		0 & 1 & 1 & 2 & 3 & 5 & 7 & 9 & 11 & 13 & 15 & 17
	\end{array}\right.\]
	\[
	\left.\begin{array}{rrrrrrrrrrr}
		-11 & -8 & -5 & -2 & 1 & -31 & -48 & -1 & -65 & -82 & 0 \\
		-9 & -14 & -19 & -24 & -29 & -3 & -2 & -53 & -1 & 0 & -82 \\
		19 & 21 & 23 & 25 & 27 & 32 & 47 & 52 & 62 & 77 & 79
	\end{array}\right)
	\]
	and their negatives. 

	We close by presenting a large example within the same family. Consider
	\[
	M_{94159} = \<94157,94159,94162\>,
	\]
	with the same parameters $a = 2$, $b = 3$, and $d = 1$. The division algorithm yields $94159 = 3138\rho + 19$.  
	Having already obtained $Gr(M_{19})$ above, we see the vast majority of $Gr(M_{94159})$ consist of the 6277 trades
	\[\al^* = (2,31384,-31385), (5,31379,-31383), \ldots, (18830,4,-18833) = \be^*\]
	all of whom differ by precisely $\hh$, as well as the 9416 trades
	\[\boldsymbol{\xi}^* = (-47078, 47076,1), (-47075, 47071,3), \ldots, (-18833,1,18831) = \boldsymbol{\omega}^*.\]
	On the other hand, $\mathcal H(\PNP{94159})$ has unchanged cardinality, consisting solely of the trades  
	\[\left(\begin{array}{rrrrr}
		0 & 1 & 3 & 47081 & 94159 \\
		-94162 & -31389 & -5 & -47081 & -94157 \\
		94159 & 31387 & 2 & 1 & 0
	\end{array}\right).\]
\end{example}


\section{Acknowledgements}


The software packages \texttt{4ti2}~\cite{4ti2}, \texttt{Normaliz}~\cite{normaliz3}, and \texttt{GAP} package \texttt{numericalsgps}~\cite{numericalsgps} provided invaluable assistance throughout this project.


\end{document}